\documentclass{amsart}
\usepackage{amssymb}
\usepackage{amsfonts}

\newtheorem{theorem}{Theorem}
\theoremstyle{plain}

\newtheorem{definition}{Definition}

\newtheorem{proposition}{Proposition}
\newtheorem{remark}{Remark}

\numberwithin{equation}{section}
\input{tcilatex}

\begin{document}
\title[$\varphi _{h,m}-$convex functions]{Inequalities via $\varphi _{h,m}-$%
convexity}
\author{M.E. \"{O}zdemir$^{\blacklozenge }$}
\address{ATATURK UNIVERSITY, K. K. EDUCATION FACULTY, DEPARTMENT OF
MATHEMATICS, 25240, ERZURUM, TURKEY}
\email{emos@atauni.edu.tr}
\author{Merve Avc\i $^{\clubsuit ,\bigstar }$}
\address{ADIYAMAN UNIVERSITY, DEPARTMENT OF MATHEMATICS, FACULTY OF SCIENCE
AND ARTS, 02040, ADIYAMAN, TURKEY}
\email{mavci@posta.adiyaman.edu.tr}
\thanks{$^{\bigstar }$Corresponding Author}

\begin{abstract}
In this paper, we define $\varphi _{h,m}-$convex functions and prove some
inequalities for this class.
\end{abstract}

\maketitle

\section{INTRODUCTION}

Let $f:I\subset 
%TCIMACRO{\U{211d} }%
%BeginExpansion
\mathbb{R}
%EndExpansion
\rightarrow 
%TCIMACRO{\U{211d} }%
%BeginExpansion
\mathbb{R}
%EndExpansion
$ be a convex function on the interval $I$ of real numbers and $a,b\in I$
with $a<b.$ The inequality%
\begin{equation}
f\left( \frac{a+b}{2}\right) \leq \frac{1}{b-a}\int_{a}^{b}f(x)dx\leq \frac{%
f(a)+f(b)}{2}  \label{hc}
\end{equation}%
is known as Hermite-Hadamard's inequality for convex functions, \cite{JFY}.

In \cite{G}, Toader defined $m-$convexity as the following.

\begin{definition}
\label{d1.1} The function $f:[0,b]\rightarrow 
%TCIMACRO{\U{211d} }%
%BeginExpansion
\mathbb{R}
%EndExpansion
,$ $b>0$, is said to be $m-$convex where $m\in \lbrack 0,1],$ if we have%
\begin{equation*}
f(tx+m(1-t)y)\leq tf(x)+m(1-t)f(y)
\end{equation*}%
for all $x,y\in \lbrack 0,b]$ and $t\in \lbrack 0,1].$ We say that $f$ is $%
m- $concave if $\left( -f\right) $ is $m-$convex.
\end{definition}

In \cite{S}, Varo\v{s}anec defined the following class of functions.

$I$ and $J$ are intervals in $%
%TCIMACRO{\U{211d} }%
%BeginExpansion
\mathbb{R}
%EndExpansion
,$ $\left( 0,1\right) \subseteq J$ and functions $h$ and $f$ are real
non-negative functions defined on $J$ and $I$, respectively.

\begin{definition}
\label{d1.6} Let $h:J\subseteq 
%TCIMACRO{\U{211d} }%
%BeginExpansion
\mathbb{R}
%EndExpansion
\rightarrow 
%TCIMACRO{\U{211d} }%
%BeginExpansion
\mathbb{R}
%EndExpansion
$ be a non-negative function, $h\neq 0.$ We say that $f:I\rightarrow 
%TCIMACRO{\U{211d} }%
%BeginExpansion
\mathbb{R}
%EndExpansion
$ is an $h-$convex function, or that $f$ belongs to the class $SX(h,I)$, if $%
f$ is non-negative and for all $x,y\in I$,$\alpha \in \left( 0,1\right) $ we
have%
\begin{equation}
f\left( \alpha x+(1-\alpha )y\right) \leq h(\alpha )f(x)+h(1-\alpha )f(y)
\label{D}
\end{equation}
\end{definition}

If inequality \ref{D} is reversed, then $f$ is said to be $h-$concave, i.e. $%
f\in SV\left( h,I\right) $.

In \cite{SSY}, Sar\i kaya et al. proved a variant of Hadamard inequality
which holds for $h-$convex functions.

\begin{theorem}
\label{t1.9} Let $f\in SX\left( h,I\right) ,$ $a,b\in I,$ with $a<b$ and $%
f\in L_{1}\left( \left[ a,b\right] \right) .$ Then
\end{theorem}

\begin{equation}
\frac{1}{2h\left( \frac{1}{2}\right) }f\left( \frac{a+b}{2}\right) \leq 
\frac{1}{b-a}\int_{a}^{b}f\left( x\right) dx\leq \left[ f\left( a\right)
+f\left( b\right) \right] \int_{0}^{1}h\left( \alpha \right) d\alpha .
\label{hh}
\end{equation}

In \cite{OAS}, \"{O}zdemir et al. defined $\left( h,m\right) -$convexity and
obtained Hermite-Hadamard-type inequalities as following .

\begin{definition}
\label{d1.7} Let $h:J\subset 
%TCIMACRO{\U{211d} }%
%BeginExpansion
\mathbb{R}
%EndExpansion
\rightarrow 
%TCIMACRO{\U{211d} }%
%BeginExpansion
\mathbb{R}
%EndExpansion
$ be a non-negative function. We say that $f:\left[ 0,b\right] \rightarrow 
%TCIMACRO{\U{211d} }%
%BeginExpansion
\mathbb{R}
%EndExpansion
$ is a $\left( h,m\right) -$convex function, if $f$ is non-negative and for
all $x,y\in \left[ 0,b\right] ,m\in \lbrack 0,1]$ and $\alpha \in (0,1),$ we
have 
\begin{equation*}
f(\alpha x+m(1-\alpha )y)\leq h(\alpha )f(x)+mh(1-\alpha )f(y).
\end{equation*}%
If the inequality is reversed, then $f$ is said to be $\left( h,m\right) -$%
concave function on $[0,b].$
\end{definition}

\begin{theorem}
\label{t1.11} Let $f:\left[ 0,\infty \right) \rightarrow 
%TCIMACRO{\U{211d} }%
%BeginExpansion
\mathbb{R}
%EndExpansion
$ be an $\left( h,m\right) -$convex function with $m\in \left( 0,1\right] ,$ 
$t\in \left[ 0,1\right] .$ If \ $0\leq a<b<\infty $ and $f\in $ $L_{1}\left[
ma,b\right] ,$ then the following inequality holds:%
\begin{eqnarray*}
&&\frac{1}{m+1}\left[ \frac{1}{mb-a}\int_{a}^{mb}f\left( x\right) dx+\frac{1%
}{b-ma}\int_{ma}^{b}f\left( x\right) dx\right]  \\
&& \\
&\leq &\left[ f(a)+f(b)\right] \int_{0}^{1}h\left( t\right) dt.
\end{eqnarray*}
\end{theorem}

Let us consider a function $\varphi :[a,b]\rightarrow \lbrack a,b]$ where $%
[a,b]\subset 
%TCIMACRO{\U{211d} }%
%BeginExpansion
\mathbb{R}
%EndExpansion
.$ In \cite{Y}, Youness defined the $\varphi -$convex functions as the
following:

\begin{definition}
\label{d1.8} A function $f:[a,b]\rightarrow 
%TCIMACRO{\U{211d} }%
%BeginExpansion
\mathbb{R}
%EndExpansion
$ is said to be $\varphi -$convex on $[a,b]$ if for every two points $x\in
\lbrack a,b],y\in \lbrack a,b]$ and $t\in \lbrack 0,1],$ the following
inequality holds:%
\begin{equation*}
f\left( t\varphi (x)+(1-t)\varphi (y)\right) \leq tf(\varphi
(x))+1-tf(\varphi (y)).
\end{equation*}
\end{definition}

In \cite{Z}, M.Z. Sarikaya defined $\varphi _{h}-$convex functions and
obtained the following inequalities for this class.

\begin{definition}
\label{d1.9} Let $I$ be an interval in $%
%TCIMACRO{\U{211d} }%
%BeginExpansion
\mathbb{R}
%EndExpansion
$ and $h:(0,1)\rightarrow (0,\infty )$ be a given function. We say that a
function $f:I\rightarrow \lbrack 0,\infty )$ is $\varphi _{h}-$convex if 
\begin{equation}
f\left( t\varphi (x)+(1-t)\varphi (y)\right) \leq h(t)f(\varphi
(x))+h(1-t)f(\varphi (y))  \label{12}
\end{equation}%
for all $x,y\in I$ and $t\in (0,1).$ If inequality (\ref{12}) is reversed,
then $f$ is said to be $\varphi _{h}-$concave.
\end{definition}

\begin{theorem}
\label{t1.13} Let $h:\left( 0,1\right) \rightarrow \left( 0,\infty \right) $
be a given function. If $f:I\rightarrow \lbrack 0,\infty )$ is Lebesgue
integrable and $\varphi _{h}-$convex for continuous function $\varphi
:[a,b]\rightarrow \lbrack a,b],$ then the following inequality holds:%
\begin{eqnarray*}
&&\frac{1}{\varphi (b)-\varphi (a)}\int_{\varphi (a)}^{\varphi
(b)}f(x)f(\varphi (a)+\varphi (b)-x)dx \\
&& \\
&\leq &\left[ f^{2}\left( \varphi (x)\right) +f^{2}\left( \varphi (y)\right) %
\right] \int_{0}^{1}h(t)h\left( 1-t\right) dt+2f\left( \varphi (x)\right)
f\left( \varphi (y)\right) \int_{0}^{1}h^{2}(t)dt.
\end{eqnarray*}
\end{theorem}

\begin{theorem}
\label{t1.14} Let $h:\left( 0,1\right) \rightarrow \left( 0,\infty \right) $
be a given function. If $f,g:I\rightarrow \lbrack 0,\infty )$ is Lebesgue
integrable and $\varphi _{h}-$convex for continuous function $\varphi
:[a,b]\rightarrow \lbrack a,b],$ then the following inequality holds:%
\begin{eqnarray*}
&&\frac{1}{\varphi \left( b\right) -\varphi \left( a\right) -}\int_{\varphi
(a)}^{\varphi (b)}f\left( x\right) g(x)dx \\
&\leq &M(a,b)\int_{0}^{1}h^{2}(t)dt+N(a,b)\int_{0}^{1}h(t)h(1-t)dt
\end{eqnarray*}%
where 
\begin{equation*}
M(a,b)=f\left( \varphi (x)\right) g\left( \varphi (x)\right) +f\left(
\varphi (y)\right) g\left( \varphi (y)\right) 
\end{equation*}%
and 
\begin{equation*}
N(a,b)=f\left( \varphi (x)\right) g\left( \varphi (y)\right) +f\left(
\varphi (y)\right) g\left( \varphi (x)\right) .
\end{equation*}
\end{theorem}

The aim of this paper is to define a new class of convex function and then
establish new Hermite-Hadamard-type inequalities.

\section{MAIN RESULTS}

In the begining we give a new definition $\varphi _{h,m}-$convex function.

$I$ and $J$ are intervals on $%
%TCIMACRO{\U{211d} }%
%BeginExpansion
\mathbb{R}
%EndExpansion
,\left( 0,1\right) \subseteq J$ and functions $h$ and $f$ are real
non-negative functions defined on $J$ and $I$, respectively.

\begin{definition}
\label{d2.1} Let $h:J\subset 
%TCIMACRO{\U{211d} }%
%BeginExpansion
\mathbb{R}
%EndExpansion
\rightarrow 
%TCIMACRO{\U{211d} }%
%BeginExpansion
\mathbb{R}
%EndExpansion
$ be a non-negative function, $h\neq 0.$ We say that $f:\left[ 0,b\right]
\subseteq \left[ 0,\infty \right) \rightarrow 
%TCIMACRO{\U{211d} }%
%BeginExpansion
\mathbb{R}
%EndExpansion
$ is a $\varphi _{h,m}-$convex function, if $f$ is non-negative and
satisfies the inequality%
\begin{equation}
f(t\varphi (x)+m(1-t)\varphi (y))\leq h(t)f(\varphi (x))+mh(1-t)f(\varphi
(y))  \label{h}
\end{equation}%
for all $x,y\in \left[ 0,b\right] ,t\in \left( 0,1\right) .$
\end{definition}

If the inequality (\ref{h}) is reversed, then $f$ is said to be $\varphi
_{h,m}-$concave function on $\left[ 0,b\right] .$

Obviously, if we choose $h(t)=t$ and $m=1$ we have non-negative $\varphi -$%
convex functions$.$ If we choose $m=1,$ then we have $\varphi _{h}-$convex
functions. If we choose $m=1$ and $\varphi (x)=x$ the two definitions $%
\varphi _{h,m}-$convex and $h-$convex functions become identical.

The following results were obtained for $\varphi _{h,m}-$convex functions.

\begin{proposition}
\label{prop 2.1} If $f,g$ are $\varphi _{h,m}-$convex functions and $\lambda
>0,$ then $f+g$ and $\lambda f$ are $\varphi _{h,m}-$convex functions.
\end{proposition}

\begin{proof}
From the definition of $\varphi _{h,m}-$convex functions we can write%
\begin{equation*}
f(t\varphi (x)+m(1-t)\varphi (y))\leq h(t)f(\varphi (x))+mh(1-t)f(\varphi
(y)))
\end{equation*}%
and 
\begin{equation*}
g(t\varphi (x)+m(1-t)\varphi (y))\leq h(t)g(\varphi (x))+mh(1-t)g(\varphi
(y)))
\end{equation*}%
for all $x,y\in \left[ 0,b\right] ,m\in (0,1]$ and $t\in \left[ 0,1\right] .$
If we add the above inequalities we get 
\begin{equation*}
\left( f+g\right) (t\varphi (x)+m(1-t)\varphi (y))\leq h(t)\left( f+g\right)
(\varphi (x))+mh(1-t)\left( f+g\right) (\varphi (y)).
\end{equation*}%
And also we have 
\begin{equation*}
\lambda f(t\varphi (x)+m(1-t)\varphi (y))\leq h(t)\lambda f(\varphi
(x))+mh(1-t)\lambda f(\varphi (y)))
\end{equation*}%
which completes the proof.
\end{proof}

\begin{proposition}
\label{prop 2.2} Let $h_{1},h_{2}:\left( 0,1\right) \rightarrow \left(
0,\infty \right) $ be functions such that $h_{2}\left( t\right) \leq
h_{1}\left( t\right) $ for all $t\in \left( 0,1\right) .$ If $f$ is $\varphi
_{h_{2},m}-$convex on $[0,b],$ then for all $x,y\in \left[ 0,b\right] $ $f$
is $\varphi _{h_{1},m}-$convex on $[0,b].$
\end{proposition}

\begin{proof}
Since $f$ is $\varphi _{h_{2},m}-$convex on $[0,b],$ for all $x,y\in \left[
0,b\right] $ and $t\in (0,1),$ we have 
\begin{eqnarray*}
f(t\varphi (x)+m(1-t)\varphi (y)) &\leq &h_{2}(t)f(\varphi
(x))+mh_{2}(1-t)f(\varphi (y))) \\
&\leq &h_{1}(t)f(\varphi (x))+mh_{1}(1-t)f(\varphi (y)))
\end{eqnarray*}%
which completes the proof.
\end{proof}

\begin{theorem}
\label{t2.0} Let $f$ be $\varphi _{h,m}-$convex function. Then i) if $%
\varphi $ is linear, then $f\circ \varphi $ is $\left( h-m\right) -$convex
and ii) if $f$ is increasing and $\varphi $ is $m-$convex, then $f\circ
\varphi $ is $\left( h-m\right) -$convex.
\end{theorem}

\begin{proof}
i) From $\varphi _{h,m}-$convexity of $f$ and linearity of $\varphi ,$ we
have%
\begin{eqnarray*}
f\circ \varphi \left[ tx+m(1-t)y\right] &=&f\left[ \varphi \left(
tx+m(1-t)y\right) \right] \\
&=&f\left[ t\varphi (x)+m(1-t)\varphi (y)\right] \\
&\leq &h(t)f\circ \varphi (x)+mh\left( 1-t\right) f\circ \varphi (y)
\end{eqnarray*}%
which completes the proof for first case.

ii) From $m-$convexity of $\varphi $, we have%
\begin{equation*}
\varphi \left[ tx+m(1-t)y\right] \leq t\varphi (x)+m(1-t)\varphi (y).
\end{equation*}%
Since $f$ is increasing we can write 
\begin{eqnarray*}
f\circ \varphi \left[ tx+m(1-t)y\right] &\leq &f\left[ t\varphi
(x)+m(1-t)\varphi (y)\right] \\
&\leq &h(t)f\circ \varphi (x)+mh\left( 1-t\right) f\circ \varphi (y).
\end{eqnarray*}%
This completes the proof for this case.
\end{proof}

\begin{theorem}
\label{t2.00} Let $h:J\subseteq 
%TCIMACRO{\U{211d} }%
%BeginExpansion
\mathbb{R}
%EndExpansion
\rightarrow 
%TCIMACRO{\U{211d} }%
%BeginExpansion
\mathbb{R}
%EndExpansion
$ be a non-negative function, $h\neq 0$ and $f:\left[ 0,b\right] \subseteq %
\left[ 0,\infty \right) \rightarrow 
%TCIMACRO{\U{211d} }%
%BeginExpansion
\mathbb{R}
%EndExpansion
$ be an $\varphi _{h,m}-$convex function with $m\in (0,1]$ and $t\in \left(
0,1\right) .$ Then for all $x,y\in \lbrack 0,b],$ the function $%
g:[0,1]\rightarrow 
%TCIMACRO{\U{211d} }%
%BeginExpansion
\mathbb{R}
%EndExpansion
,$ $g(t)=f(t\varphi (x)+m(1-t)\varphi (y))$ is $\left( h-m\right) -$convex
on $\left[ 0,b\right] .$
\end{theorem}

\begin{proof}
Since $f$ is $\varphi _{h,m}-$convex function, for $x,y\in \lbrack 0,b],$ $%
\lambda _{1},\lambda _{2}\in \left( 0,1\right) $ with $\lambda _{1}+\lambda
_{2}=1$ and $t_{1},t_{2}\in \left( 0,1\right) $ we obtain%
\begin{eqnarray*}
&&g\left( \lambda _{1}t_{1}+m\lambda _{2}t_{2}\right)  \\
&=&f\left[ \left( \lambda _{1}t_{1}+m\lambda _{2}t_{2}\right) \varphi
(x)+m\left( 1-\lambda _{1}t_{1}-m\lambda _{2}t_{2}\right) \varphi (y)\right] 
\\
&=&f\left[ \lambda _{1}\left( t_{1}\varphi (x)+m\left( 1-t_{1}\right)
\varphi (y)\right) +m\lambda _{2}\left( t_{2}\varphi (x)+m\left(
1-t_{2}\right) \varphi (y)\right) \right]  \\
&\leq &h\left( \lambda _{1}\right) f\left( t_{1}\varphi (x)+m\left(
1-t_{1}\right) \varphi (y)\right) +mh\left( \lambda _{2}\right) f\left(
t_{2}\varphi (x)+m\left( 1-t_{2}\right) \varphi (y)\right)  \\
&=&h\left( \lambda _{1}\right) g\left( t_{1}\right) +mh\left( \lambda
_{2}\right) g\left( t_{2}\right) 
\end{eqnarray*}%
which shows the $\left( h-m\right) -$convexity of $g.$
\end{proof}

\begin{theorem}
\label{t2.1} Let $h:J\subseteq 
%TCIMACRO{\U{211d} }%
%BeginExpansion
\mathbb{R}
%EndExpansion
\rightarrow 
%TCIMACRO{\U{211d} }%
%BeginExpansion
\mathbb{R}
%EndExpansion
$ be a non-negative function, $h\neq 0$ and $f:\left[ 0,b\right] \subseteq %
\left[ 0,\infty \right) \rightarrow 
%TCIMACRO{\U{211d} }%
%BeginExpansion
\mathbb{R}
%EndExpansion
$ be a $\varphi _{h,m}-$convex function with $m\in (0,1]$ and $t\in \left(
0,1\right) .$ If $f\in L_{1}\left[ \varphi (a),m\varphi (b)\right] ,$ $h\in
L_{1}\left[ 0,1\right] ,$ one has the following inequality: 
\begin{eqnarray*}
&&\frac{1}{m\varphi (y)-\varphi (x)}\int_{\varphi (x)}^{m\varphi
(y)}f(u)f(\varphi (x)+m\varphi (y)-u)du \\
&& \\
&\leq &f^{2}\left( \varphi (x)\right) +m^{2}f^{2}\left( \varphi (y)\right)
\int_{0}^{1}h(t)h\left( 1-t\right) dt+f\left( \varphi (x)\right) f\left(
\varphi (y)\right) \left[ m+1\right] \int_{0}^{1}h^{2}(t)dt
\end{eqnarray*}
\end{theorem}

\begin{proof}
Since $f$ is $\varphi _{h,m}-$convex function, $t\in \left[ 0,1\right] $ and 
$m\in (0,1],$ then%
\begin{equation*}
f(t\varphi (x)+m(1-t)\varphi (y))\leq h(t)f(\varphi (x))+mh(1-t)f(\varphi
(y))
\end{equation*}
and%
\begin{equation*}
f((1-t)\varphi (x)+mt\varphi (y))\leq h(1-t)f(\varphi (x))+mh(t)f(\varphi
(y))
\end{equation*}
for all $x,y\in \lbrack 0,b].$

By multiplying these inequalities and integrating on $\left[ 0,1\right] $
with respect to $t$, we obtain%
\begin{eqnarray*}
&&\int_{0}^{1}f(t\varphi (x)+m(1-t)\varphi (y))f((1-t)\varphi (x)+mt\varphi
(y))dt \\
&\leq &f^{2}\left( \varphi (x)\right) \int_{0}^{1}h(t)h(1-t)dt+mf\left(
\varphi (x)\right) f\left( \varphi (y)\right) \int_{0}^{1}h^{2}(t)dt \\
&&+mf\left( \varphi (x)\right) f\left( \varphi (y)\right)
\int_{0}^{1}h^{2}(1-t)dt+m^{2}f^{2}\left( \varphi (y)\right)
\int_{0}^{1}h(t)h(1-t)dt \\
&=&\left[ f^{2}\left( \varphi (x)\right) +m^{2}f^{2}\left( \varphi
(y)\right) \right] \int_{0}^{1}h(t)h(1-t)dt+f\left( \varphi (x)\right)
f\left( \varphi (y)\right) \left[ m+1\right] \int_{0}^{1}h^{2}(t)dt.
\end{eqnarray*}%
If we change the variable $u=t\varphi (x)+m(1-t)\varphi (y),$ we obtain the
inequality which is the required.
\end{proof}

\begin{remark}
\label{rem 2.0} In Theorem \ref{t2.1}, if we choose $m=1$ Theorem \ref{t2.1}
reduces to Theorem \ref{t1.13}.
\end{remark}

\begin{theorem}
\label{t2.2.} Under the assumptions of Theorem \ref{t2.1}, we have the
following inequality%
\begin{equation*}
\frac{1}{m\varphi \left( y\right) -\varphi \left( x\right) }\int_{\varphi
(x)}^{m\varphi (y)}f\left( u\right) du\leq \left[ f\left( \varphi (x)\right)
+f\left( \varphi (y)\right) \right] \int_{0}^{1}h(t)dt.
\end{equation*}
\end{theorem}

\begin{proof}
By definition of $\varphi _{h,m}-$convex function we can write 
\begin{equation*}
f(t\varphi (x)+m(1-t)\varphi (y))\leq h(t)f(\varphi (x))+mh(1-t)f(\varphi
(y)).
\end{equation*}%
If we integrate the above inequality on $\left[ 0,1\right] $ with respect to 
$t$ and change the variable $u=t\varphi (x)+m(1-t)\varphi (y),$ we obtained
the required inequality.
\end{proof}

\begin{remark}
\label{rem 2.1} In Theorem \ref{t2.2.}, if we choose $m=1$ and $\varphi
:[a,b]\rightarrow \lbrack a,b],\varphi (x)=x$, we obtained the inequality
which is the right hand side of (\ref{hh}).
\end{remark}

\begin{theorem}
\label{t2.2} Under the assumptions of Theorem \ref{t2.1}, we have the
following inequality 
\begin{eqnarray*}
&&\frac{1}{m+1}\left[ \frac{1}{\varphi \left( y\right) -m\varphi \left(
x\right) }\int_{m\varphi (x)}^{\varphi (y)}f\left( u\right) du+\frac{1}{%
m\varphi \left( y\right) -\varphi \left( x\right) }\int_{\varphi
(x)}^{m\varphi (y)}f\left( u\right) du\right]  \\
&\leq &\left[ f\left( \varphi (x)\right) +f\left( \varphi (y)\right) \right]
\int_{0}^{1}h(t)dt
\end{eqnarray*}%
for all $0\leq m\varphi \left( x\right) \leq \varphi \left( x\right) \leq
m\varphi (y)<\varphi (y)<\infty .$
\end{theorem}

\begin{proof}
Since $f$ is $\varphi _{h,m}-$convex function, we can write 
\begin{eqnarray*}
f(t\varphi (x)+m(1-t)\varphi (y)) &\leq &h(t)f(\varphi (x))+mh(1-t)f(\varphi
(y)), \\
&& \\
f((1-t)\varphi (x)+mt\varphi (y)) &\leq &h(1-t)f(\varphi (x))+mh(t)f(\varphi
(y)), \\
&& \\
f(t\varphi (y)+m(1-t)\varphi (x)) &\leq &h(t)f(\varphi (y))+mh(1-t)f(\varphi
(x)),
\end{eqnarray*}%
and%
\begin{equation*}
f((1-t)\varphi (y)+mt\varphi (x))\leq h(1-t)f(\varphi (y))+mh(t)f(\varphi
(x)).
\end{equation*}

By summing these inequalities and integrating on $\left[ 0,1\right] $ with
respect to $t$, we obtain%
\begin{eqnarray*}
&&\int_{0}^{1}f(t\varphi (x)+m(1-t)\varphi (y))dt+\int_{0}^{1}f((1-t)\varphi
(x)+mt\varphi (y))dt \\
&&+\int_{0}^{1}f(t\varphi (y)+m(1-t)\varphi
(x))dt+\int_{0}^{1}f((1-t)\varphi (y)+mt\varphi (x))dt \\
&\leq &\left[ f\left( \varphi (x)\right) +f\left( \varphi (y)\right) \right]
\left( m+1\right) \left[ \int_{0}^{1}h(t)dt+\int_{0}^{1}h(1-t)dt\right] .
\end{eqnarray*}%
It is easy to see that%
\begin{eqnarray*}
\int_{0}^{1}f(t\varphi (x)+m(1-t)\varphi (y))dt
&=&\int_{0}^{1}f((1-t)\varphi (y)+mt\varphi (x))dt=\frac{1}{m\varphi \left(
y\right) -\varphi \left( x\right) }\int_{\varphi (x)}^{m\varphi (y)}f\left(
u\right) du, \\
\int_{0}^{1}f(t\varphi (y)+m(1-t)\varphi (x))dt
&=&\int_{0}^{1}f((1-t)\varphi (y)+mt\varphi (x))dt=\frac{1}{\varphi \left(
y\right) -m\varphi \left( x\right) }\int_{m\varphi (x)}^{\varphi (y)}f\left(
u\right) du
\end{eqnarray*}%
and%
\begin{equation*}
\int_{0}^{1}h(t)dt=\int_{0}^{1}h(1-t)dt.
\end{equation*}%
If we write these equalities in the above inequality we obtain the required
result.
\end{proof}

\begin{remark}
\label{rem 2.2} In Theorem \ref{t2.2}, if we choose $\varphi
:[a,b]\rightarrow \lbrack a,b],\varphi (x)=x$ Theorem \ref{t2.2} reduces to
Theorem \ref{t1.11}..
\end{remark}

\begin{theorem}
\label{t2.3}Let $h:J\subseteq 
%TCIMACRO{\U{211d} }%
%BeginExpansion
\mathbb{R}
%EndExpansion
\rightarrow 
%TCIMACRO{\U{211d} }%
%BeginExpansion
\mathbb{R}
%EndExpansion
$ be a non-negative function, $h\neq 0$ and $f,g:\left[ 0,b\right] \subseteq %
\left[ 0,\infty \right) \rightarrow 
%TCIMACRO{\U{211d} }%
%BeginExpansion
\mathbb{R}
%EndExpansion
$ be $\varphi _{h,m}-$convex functions with $m\in (0,1]$. If $f$ and $g$ are
Lebesque integrable, the following inequality holds: 
\begin{eqnarray*}
&&\frac{1}{m\varphi \left( y\right) -\varphi \left( x\right) -}\int_{\varphi
(x)}^{m\varphi (y)}f\left( u\right) g(u)du \\
&\leq &M(a,b)\int_{0}^{1}h^{2}(t)dt+mN(a,b)\int_{0}^{1}h(t)h(1-t)dt
\end{eqnarray*}%
where 
\begin{equation*}
M(a,b)=f\left( \varphi (x)\right) g\left( \varphi (x)\right) +m^{2}f\left(
\varphi (y)\right) g\left( \varphi (y)\right) 
\end{equation*}%
and 
\begin{equation*}
N(a,b)=f\left( \varphi (x)\right) g\left( \varphi (y)\right) +f\left(
\varphi (y)\right) g\left( \varphi (x)\right) .
\end{equation*}
\end{theorem}

\begin{proof}
Since $f$ and $g$ are $\varphi _{h,m}-$convex functions, we can write%
\begin{equation*}
f(t\varphi (x)+m(1-t)\varphi (y))\leq h(t)f(\varphi (x))+mh(1-t)f(\varphi
(y))
\end{equation*}%
and 
\begin{equation*}
g(t\varphi (x)+m(1-t)\varphi (y))\leq h(t)g(\varphi (x))+mh(1-t)g(\varphi
(y)).
\end{equation*}%
If we multiply the above inequalities and integrate on $\left[ 0,1\right] $
with respect to $t$, we obtain%
\begin{eqnarray*}
&&\int_{0}^{1}f(t\varphi (x)+m(1-t)\varphi (y))g(t\varphi (x)+m(1-t)\varphi
(y))dt \\
&\leq &f\left( \varphi (x)\right) g\left( \varphi (x)\right)
\int_{0}^{1}h^{2}(t)dt+m^{2}f\left( \varphi (y)\right) g\left( \varphi
(y)\right) \int_{0}^{1}h^{2}(1-t)dt \\
&&+m\left[ f\left( \varphi (x)\right) g\left( \varphi (y)\right) +f\left(
\varphi (y)\right) g\left( \varphi (x)\right) \right] \int_{0}^{1}h\left(
t\right) h(1-t)dt.
\end{eqnarray*}%
If we change the variable $u=t\varphi (x)+m(1-t)\varphi (y),$ we obtain the
inequality which is the required.
\end{proof}

\begin{remark}
\label{rem 2.3} In Theorem \ref{t2.3}, if we choose $m=1$ Theorem \ref{t2.3}
reduces to Theorem \ref{t1.14}.
\end{remark}

\end{document}